\documentclass[11pt,english]{amsart}
\DeclareSymbolFontAlphabet{\mathcal}{symbols}
\usepackage[cmtip,all]{xy}
\usepackage[a4paper]{geometry}
\numberwithin{equation}{section}
\usepackage{amsmath}
\usepackage[psamsfonts]{amssymb}
\usepackage[psamsfonts]{eucal}
\usepackage{amsthm}
\usepackage[psamsfonts]{amssymb}
\usepackage{mathrsfs}
\usepackage{framed}
\usepackage{hyperref}
\hypersetup{
    colorlinks,
    citecolor=blue,
    filecolor=black,
    linkcolor=blue,
    urlcolor=black
}
\usepackage{enumerate}
\usepackage{graphicx}
\usepackage{tikz}
\usepackage{color}
\definecolor{trama}{gray}{.875}
\usepackage{lmodern}
\usepackage{textcomp}
\usepackage{url}

\newtheorem{theorem}{Theorem}[section]

\newtheorem{proposition}[theorem]{Proposition} %
\newtheorem{corollary}[theorem]{Corollary}

\theoremstyle{definition}
\newtheorem{definition}[theorem]{Definition}
\newtheorem{example}[theorem]{Example}

\theoremstyle{remark}
\newtheorem{remark}[theorem]{Remark}

\numberwithin{equation}{section}

\newcommand{\secref}[1]{Section~\ref{#1}}

\newcommand{\thmref}[1]{Theorem~\ref{#1}}
\newcommand{\propref}[1]{Proposition~\ref{#1}}

\newcommand{\corref}[1]{Corollary~\ref{#1}}

\newcommand{\remref}[1]{Remark~\ref{#1}}
\newcommand{\exemref}[1]{Example~\ref{#1}}

\newcommand{\defref}[1]{Definition~\ref{#1}}

\def\ov{\overline}
\newcommand{\cU}{\mathcal U}
\newcommand{\cD}{\mathcal D}

\def\gd{{\mathfrak{d}}}
\def\gC{{\mathfrak{C}}}
\def\gH{{\mathfrak{H}}}

\def\tc{{\mathtt c}}

\def\tv{{\mathtt v}}

\def\N{\mathbb{N}}
\def\R{\mathbb{R}}
\def\Z{\mathbb{Z}}
\def\im{{\rm Im\,}}

\def\codim{{\rm codim\,}}
\def\pr{{\rm pr}}

\def\rc{{\mathring{\tc}}}
\def\cS{{\mathcal S}}

 \def\1{{\mathbf 1}}
   
   \def\tN{{\widetilde{N}}}
\def\crH{{\mathscr H}}
\def\gQ{{\mathfrak{Q}}}

\def\tto{{\mathtt{o}}}

 \newcommand{\menos}{\backslash}

\newcommand{\dos}[2]{{#1}_{_{#2}}}

\newcommand{\Hiru}[3]{{#1}^{^{#2}}{( #3 )}}

\newcommand{\lau}[4]{{#1}^{^{#2}}_{_{#3}}{( #4 )}}

\title[Borel-Moore Intersection Homology]{Poincar\'e duality, cap product and Borel-Moore Intersection Homology }

\date{\today}

\author{Martintxo Saralegi-Aranguren}
\address{Laboratoire de Math{\'e}matiques de Lens\\  
      EA 2462 \\
      Universit\'e d'Artois\\
         SP18, rue Jean Souvraz\\
          62307 Lens Cedex\\
         France}
\email{martin.saraleguiaranguren@univ-artois.fr}

\author{Daniel Tanr\'e}
\address{D\'epartement de Math{\'e}matiques\\
         UMR-CNRS 8524 \\
         Universit\'e de Lille\\
         59655 Villeneuve d'Ascq Cedex\\
         France}
\email{Daniel.Tanre@univ-lille.fr}
\thanks{
The second author was  supported by the MINECO and FEDER research project MTM2016-78647-P. 
and the ANR-11-LABX-0007-01  ``CEMPI''}

\subjclass[2010]{55N33, 57P10, 57N80, 55U30}

\keywords{Intersection homology; Borel-Moore chains; cap product;  Poincar\'e duality}
 
\commby{}

\begin{document} 

\begin{abstract}
Using a cap product, we construct an explicit Poincar\'e 
duality isomorphism between the blown-up intersection
cohomology and the Borel-Moore intersection homology, for any commutative ring of coefficients and 
second-countable, oriented pseudomanifolds.
\end{abstract}

\maketitle

\section*{Introduction}
Poincar\'e duality of singular spaces is the ``raison d'\^etre'' (\cite[Section 8.2]{FriedmanBook}) 
of intersection homology.
It has been proven by Goresky and MacPherson in their first paper on intersection homology
(\cite{GM1}) for compact PL pseudomanifolds and rational coefficients
and extended to $\Z$ coefficients with some hypothesis on the torsion part, by Goresky and Siegel in \cite{GS}.
Friedman and McClure obtain this isomorphism for a topological pseudomanifold, from a cap product with a fundamental class,
for any field of coefficients in \cite{FM}, see also \cite{FriedmanBook} for a commutative ring 
of coefficients with restrictions on the torsion.

\medskip
Using the blown-up intersection cohomology with compact supports, 
we have established in \cite{CST2} a Poincar\'e duality for any
commutative ring of coefficients, without hypothesis on the torsion part, for any oriented paracompact pseudomanifold.
Moreover, we also set up in \cite{CST5} a Poincar\'e duality between the
blown-up intersection cohomology and the Borel-Moore intersection homology of an oriented  PL pseudomanifold $X$. 

\smallskip
This paper is the ``cha\^inon manquant:'' the existence of an explicit  Poincar\'e duality isomorphism 
between the
blown-up intersection cohomology and the Borel-Moore intersection homology, from a 
cap product with the fundamental class, for any commutative ring of coefficients
 and any second-countable, oriented pseudomanifold. 
 This allows the definition of an intersection product on the Borel-Moore intersection homology, 
 induced from the Poincar\'e duality and a cup product, as in the case of manifolds, see
 \corref{cor:intersectionproduct} and \remref{rem:lastone}.
 Let us note that  a Poincar\'e duality isomorphism cannot 
 exist with an intersection cohomology defined as homology of the dual complex of intersection chains,
 see \exemref{exam:pasdual}.

\smallskip
In \secref{sec:back}, we recall basic background on pseudomanifolds and intersection homology. In particular, 
we present the complex of blown-up cochains, already introduced and studied in a series of papers 
\cite{ CST6,CST7, CST4, CST1,CST2,CST5,CST3} (also called Thom-Whitney cochains in some works).
\secref{sec:BM} contains the main properties of Borel-Moore intersection homology: the existence of 
a Mayer-Vietoris exact sequence in \thmref{thm:MV} and the recollection of some results established in \cite{CST5}.
\secref{sec:Poincare} is devoted to the proof of the main result stated in \thmref{thm:dual}:
the existence of an isomorphism
between the blown-up intersection cohomology and the 
Borel-Moore intersection homology, by using the cap-product with the fundamental
class of a second-countable, oriented pseudomanifold, 
$$\cD_{X}\colon \crH^*_{\ov{p}}(X;R)\to \gH^{\infty,\ov{p}}_{n-*}(X;R),$$
for any commutative ring of coefficients.
In \cite{CST5}, we prove it for PL-pseudomanifolds with a sheaf presentation of  intersection homology.
Here the duality is realized for topological pseudomanifolds,
by a map defined at  the chain complexes level by the cap product 
with a cycle representing the fundamental class.

\smallskip
Notice also that  the intersection homology of this work (\defref{def:lessimplexes}) 
is a general version, called tame homology in 
\cite{CST3} and non-GM in \cite{FriedmanBook}, which  coincides with the original one   for the perversities of \cite{GM1}.
Let us  also observe that
our definition of Borel-Moore intersection homology coincides with the
one studied by G. Friedman in \cite{MR2276609} for  perversities depending
only on the codimension of strata.

\medskip
Homology and cohomology are considered with coefficients in a commutative ring, $R$. In general, we do not
explicit them in the proofs. 
For any topological space $X$, we denote by $\tc X=X\times [0,1]/X\times\{0\}$
the cone on $X$
and by $\rc X=X\times [0,1[/X\times\{0\}$ the open cone on $X$.
Elements of the cones $\tc X$ and $\rc X$ are denoted $[x,t]$ and the apex is $\tv=[-,0]$.

\tableofcontents

\section{Background}\label{sec:back}

\subsection{Pseudomanifold}
In \cite{GM1}, M. Goresky and R. MacPherson introduce intersection homology for the study of pseudomanifolds.
Some basic properties of intersection homology, as the existence of a Mayer-Vietoris sequence, 
do not require such a structure and exist for filtered spaces.

\begin{definition}
A \emph{filtered space of (formal) dimension $n$,} $X$, is a Hausdorff space
endowed with a filtration by closed subsets,
$$\emptyset=X_{-1}\subseteq X_0\subseteq X_1\subseteq\dots\subseteq X_n=X,$$
such that $X_n\backslash X_{n-1}\neq \emptyset$.
The \emph{strata} of $X$ of dimension $i$ are the connected components $S$ of $X_{i}\backslash X_{i-1}$; we
denote $\dim S=i$ and $\codim S=\dim X-\dim S$. The \emph{regular strata} are the strata of dimension $n$
and  the \emph{singular set} is the subspace $\Sigma =X_{n-1}$. We denote by $\cS_{X}$ (or $\cS$ if there is no ambiguity) the set of non-empty strata.
\end{definition}

\begin{definition}
An $n$-dimensional \emph{CS set}  is a 
filtered space of dimension $n$, $X$,
such that, for any $i\in\{0,\dots,n\}$, 
$X_i\backslash X_{i-1}$ is an $i$-dimensional topological manifold or the empty set. Moreover, for each point
$x \in X_i \backslash X_{i-1}$, $i\neq n$, there exist
\begin{enumerate}[(i)]
\item an open neighborhood $V$ of $x$ in $X$, endowed with the induced filtration,
\item an open neighborhood $U$ of $x$ in  $X_i\backslash X_{i-1}$, 
\item a compact filtered space $L$  of dimension $n-i-1$, whose cone $\rc L$ is endowed with the filtration
$(\rc L)_{i}=\rc L_{i-1}$, 
\item a homeomorphism, $\varphi \colon U \times \rc L\to V$, 
such that
\begin{enumerate}[(a)]
\item $\varphi(u,\tv)=u$, for any $u\in U$, where $\tv$  is the apex of the cone $\rc L$,
\item $\varphi(U\times \rc L_{j})=V\cap X_{i+j+1}$, for all $j\in \{0,\dots,n-i-1\}$.
\end{enumerate}
\end{enumerate}
The filtered space $L$ is called a  \emph{link} of $x$. 
\end{definition}

Except in the reminder of a previous result (see \propref{prop:supersuperbredon}), 
this work is only concerned with particular CS sets, the pseudomanifolds.

\begin{definition}
An $n$-dimensional \emph{pseudomanifold}  is an $n$-dimensional CS set for which the link
of a point $x\in X_i\backslash X_{i-1}$, $i\neq n$, is 
a compact pseudomanifold $L$  of dimension $n-i-1$.
\end{definition}

For a pseudomanifold, the formal dimension of the underlying filtered space coincides with 
the classical dimension of the manifold $X_{n}\backslash X_{n-1}$.
In \cite{GM1}, the pseudomanifolds are supposed without strata of codimension~1. 
Here, we do not require  this property.
The class of pseudomanifolds is large anough to include 
(\cite{GM2}) complex algebraic or analytic varieties, 
real analytic varieties, Whitney and  Thom-Mather stratified spaces, quotients of manifolds by
compact Lie groups acting smootly, Thom spaces of vector bundles over triangulable compact manifolds,
suspension of manifolds, ... 

\begin{remark}\label{rem:plentyofdef}
For the convenience of the reader, we first collect basic topological definitions. A topological space
$X$  is said
\begin{enumerate}
\item \emph{separable}  if it contains a countable, dense subset;
\item \emph{second-countable} if its topology has a numerable basis; that is
there exists some numerable collection $\cU = \{U_{j} |j \in \N\}$ of open subsets such that
any open subset of $X$ can be written as a union of elements of some subfamily of $\cU$;
\item \emph{hemicompact} if it is locally compact and there exists a numerable sequence of 
relatively compact open subsets,
$(U_{i})_{i\in \N}$, such that $\ov{U}_{i}\subset U_{i+1}$
and
$X=\cup_{i}U_{i}$.
\end{enumerate}
To relate hypotheses of some following results to ones of previous works, 
we list some interactions between these  notions.
\begin{itemize}
\item A second countable space is separable and, in a metric space, the two properties are equivalent
(\cite[Theorem 16.9]{Wil}).
\item A space is locally compact and second-countable if, and only if, it is metrizable and hemicompact, see 
\cite[Corollaire of Proposition 16]{MR0173226}.
\end{itemize}
\end{remark}

As second-countability is a hereditary property, any open subset of a second-countable pseudomanifold 
is one also.  Moreover, we also know that
a second-countable pseudomanifold  is 
paracompact, separable, metrizable and hemicompact.

\subsection{Intersection homology}
We consider intersection homology relatively to the  general perversities defined  in \cite{MacPherson90}.

\begin{definition}\label{def:perversite} 
A  \emph{perversity on a filtered space,} $X$, is an  application,
$\ov{p}\colon \cS\to \Z$, defined on the set of strata of $X$ and taking the value~0 on the regular strata.
Among them, mention the null perversity $\ov{0}$, constant with value~0, and the \emph{top perversity} defined by
$\ov{t}(S)=\codim S-2$ on singular strata. 
(This infers $\ov{t}(S)=-1$ for codimension~1 strata.)
For any perversity, $\ov{p}$, the perversity $D\ov{p}:=\ov{t}-\ov{p}$ is called the
\emph{complementary perversity} of $\ov{p}$.

The pair $(X,\ov{p})$ is called a \emph{perverse space}.
For a pseudomanifold we say \emph{perverse pseudomanifold.}
\end{definition}

\begin{example}
Let $(X,\ov{p})$ be a perverse space of dimension $n$.
\begin{itemize}
\item An \emph{open perverse subspace $(U,\ov{p})$}  is an open subset $U$ of $X$,
endowed with the induced filtration and a  perversity still denoted $\ov{p}$ and  defined as follows:
if $S\subset U$ is a stratum of  $U$, such that $S\subset U\cap S'$ with $S'$ a stratum of
$X$, then $\ov{p}(S)=\ov{p}(S')$.
In the case of a perverse pseudomanifold,  $(U,\ov{p})$ is one also.
\item If $M$ is a connected topological manifold, the product  $M\times X$ is a filtered space
for the \emph{product filtration,}
$\left(M \times X\right) _i = M \times X_{i}$. 
The perversity $\ov p$ induces a perversity on $M\times X$, still denoted  $\ov p$
and  defined by $\ov p(M \times S) = \ov p(S)$ for each stratum $S$ of $X$.
\item If $X$ is compact, the open cone 
$\rc X $ is endowed with the 
\emph{conical filtration,} 
$\left(\rc X\right) _i =\rc X_{i-1}$,  $0\leq i\leq n+1$, where
$\rc \,\emptyset=\{ \tv \}$ is the apex of the cone.
A perversity $\ov{p}$ on   $\rc X$ induces a perversity on $X$ still denoted  $\ov{p}$ and defined by
$\ov{p}(S)=\ov{p}(S\times ]0,1[)$ for each stratum $S$ of $X$.
\end{itemize}
\end{example}

\medskip
\emph{For the rest of this section, we consider a perverse space $(X,\ov{p})$.}
 We introduce now a chain complex  giving  the intersection homology with coefficients in $R$, cf. \cite{CST3}.

\begin{definition}\label{def:regularsimplex}
A  \emph{regular simplex} is a continuous map $\sigma\colon\Delta\to X$
of domain an Euclidean simplex decomposed in joins,
$\Delta=\Delta_{0}\ast\Delta_{1}\ast\dots\ast\Delta_{n}$,
such that
$\sigma^{-1}X_{i} =\Delta_{0}\ast\Delta_{1}\ast\dots\ast\Delta_{i}$
for all~$i \in \{0, \dots, n\}$ and
$\Delta_n \ne \emptyset$.

Given an Euclidean regular simplex $\Delta = \Delta_0 * \dots *\Delta_n$, we denote  
$\gd\Delta$ the regular part of the chain $\partial \Delta$.  
More precisely, we set 
$\gd \Delta =\partial (\Delta_0 * \dots * \Delta_{n-1})* \Delta_n$, if $\dim(\Delta_n) = 0 $
and
 $\gd \Delta = \partial \Delta$,
if $\dim(\Delta_n)\geq 1$.
For any regular simplex $\sigma \colon\Delta \to X$, we set  $\gd \sigma=\sigma_* \circ \gd$.
 Notice that $\gd^2=0$.
 We denote by $\gC_{*}(X;R)$ 
 the  complex of  linear combinations of regular  simplices (called finite chains)
with the differential~$\gd$.
\end{definition}

 \begin{definition}\label{def:lessimplexes}
 The \emph{perverse degree of} a regular simplex $\sigma\colon\Delta=\Delta_{0}\ast \dots\ast\Delta_{n} \to X$
  is the $(n+1)$-uple,
$\|\sigma\|=(\|\sigma\|_0,\dots,\|\sigma\|_n)$,  
where 
 $\|\sigma\|_{i}=
 \dim \sigma^{-1}X_{n-i}=\dim (\Delta_{0}\ast\dots\ast\Delta_{n-i})$, 
 with the  convention $\dim \emptyset=-\infty$.
For each stratum $S$  of $X$, the \emph{perverse degree of $\sigma$ along $S$} is defined by
$$\|\sigma\|_{S}=\left\{
 \begin{array}{cl}
 -\infty,&\text{if } S\cap \sigma(\Delta)=\emptyset,\\
 \|\sigma\|_{\codim S},&\text{otherwise.}
  \end{array}\right.
$$
 A \emph{regular simplex is $\ov{p}$-allowable} if
   \begin{equation*} 
  \|\sigma\|_{S}\leq \dim \Delta-\codim S+\ov{p}(S),
  \end{equation*}
   for any stratum $S$.
   A finite chain $\xi$ is 
   \emph{$\ov{p}$-allowable} if it is a linear combination of $\ov{p}$-allowable simplices
   and of  \emph{$\ov{p}$-intersection} if  $\xi$ and its boundary $\gd \xi$ are  $\ov{p}$-allowable.
   We denote by $\gC_{*}^{\ov{p}}(X;R)$ 
the complex  of  $\ov{p}$-intersection chains and by 
$\gH_{*}^{\ov{p}}(X;R)$
 its homology, called \emph{$\ov{p}$-intersection homology}.
 \end{definition}
 
If $(U,\ov{p})$ is an open perverse subspace  of $(X,\ov{p})$, we define the
\emph{complex of relative   $\ov{p}$-intersection chains} as the quotient
$\gC_{*}^{\ov{p}}(X,U;R)=
\gC_{*}^{\ov{p}}(X;R)/
\gC_{*}^{\ov{p}}(U;R)$.
Its homology  is denoted $\gH_{*}^{\ov{p}}(X,U;R)$.
Finally, if $K\subset U$ is compact,
we have 
$\gH_{*}^{\ov{p}}(X,X\menos K;R)=
\gH_{*}^{\ov{p}}(U,U\menos K;R)
$ by excision, cf. \cite[Corollary  4.5]{CST3}.

\begin{remark}
This homology is called tame intersection homology in \cite{CST3}. 
As we are only using it in this work, for sake of simplicity, we call it  intersection homology. 
It coincides with the non-GM intersection homology of  \cite{FriedmanBook} (see \cite[Theorem~B]{CST3}) and with
 intersection homology for the original perversities of \cite{GM1}, see \cite[Remark 3.9]{CST3}.

 The GM-perversities introduced by Goresky and MacPherson in \cite{GM1} depend only on the codimension
 of the strata and verify a growth condition,
 $\ov{p}(k)\leq \ov{p}(k+1)\leq \ov{p}(k)+1$,
 that implies the topological invariance of the intersection homology groups, if there is no strata of codimension~1. In
 \cite{CST3}, we prove that a certain topological invariance is still verified within the framework of strata dependent
 perversities. For that, we consider the intrinsic space, $X^*$, associated to a pseudomanifold $X$
 (introduced by King in \cite{MR800845})
 and use  pushforward and pullback perversities. 
 In particular, if there is no singular strata in $X$ becoming regular in $X^*$, we establish the invariance of
 intersection homology under a refinement of the stratification
 (\cite[Remark 6.14]{CST3}).
\end{remark}

\subsection{Blown-up intersection cohomology}
Let 
$N_{*}(\Delta)$ and $N^*(\Delta)$
be the simplicial chains and cochains, with coefficient in $R$, of an Euclidean 
simplex $\Delta$. 
Given a face $F$  of $\Delta$, we write $\1_{F}$ the element of $N^*(\Delta)$ taking the value 1 on $F$ and 0 otherwise. 
We denote also by $(F,0)$ the same face viewed as face of the cone $\tc\Delta=[\tv]\ast \Delta $ and by $(F,1)$ 
the face $\tc F$ of $\tc \Delta$. 
The apex is denoted $(\emptyset,1)=\tc \emptyset =[\tv]$.

If $\Delta=\Delta_{0}\ast\dots\ast\Delta_{n}$ is a regular Euclidean simplex, we set
$$\tN^*(\Delta)=N^*(\tc \Delta_{0})\otimes\dots\otimes N^*(\tc \Delta_{n-1})\otimes N^*(\Delta_{n}).$$
A basis of $\tN^*(\Delta)$ is made of the elements 
$\1_{(F,\varepsilon)}=\1_{(F_{0},\varepsilon_{0})}\otimes\dots\otimes \1_{(F_{n-1},\varepsilon_{n-1})}\otimes \1_{F_{n}}$,
 where 
$\varepsilon_{i}\in\{0,1\}$ and
$F_{i}$ is a face of $\Delta_{i}$ for $i\in\{0,\dots,n\}$ or the empty set with $\varepsilon_{i}=1$ if $i<n$.
We set
$|\1_{(F,\varepsilon)}|_{>s}=\sum_{i>s}(\dim F_{i}+\varepsilon_{i})$,
with $\varepsilon_{n}=0$.

\begin{definition}\label{def:degrepervers}
Let $\ell\in \{1,\ldots,n\}$.
The  \emph{$\ell$-perverse degree} of 
$\1_{(F,\varepsilon)}\in \tN^*(\Delta)$ is
$$
\|\1_{(F,\varepsilon)}\|_{\ell}=\left\{
\begin{array}{ccl}
-\infty&\text{if}
&
\varepsilon_{n-\ell}=1,\\
|\1_{(F,\varepsilon)}|_{> n-\ell}
&\text{if}&
\varepsilon_{n-\ell}=0.
\end{array}\right.$$
For a cochain $\omega = \sum_b\lambda_b \ \1_{(F_b,\varepsilon_b) }\in\tN^*(\Delta)$ with 
$\lambda_{b}\neq 0$ for all $b$,
the \emph{$\ell$-perverse degree} is
$$\|\omega \|_{\ell}=\max_{b}\|\1_{(F_b,\varepsilon_b)}\|_{\ell}.$$
By convention, we set $\|0\|_{\ell}=-\infty$.
\end{definition}

Let $(X,\ov{p})$ be a perverse space and $\sigma\colon \Delta=\Delta_{0}\ast\dots\ast\Delta_{n}\to X$  a regular simplex.
We set $\tN^*_{\sigma}=\tN^*(\Delta)$.
Let $\delta_{\ell}\colon \Delta' 
\to\Delta$  be  the inclusion of a face, we set
$\partial_{\ell}\sigma=\sigma\circ\delta_{\ell}\colon \Delta'\to X$ with the induced filtration 
$\Delta'=\Delta'_{0}\ast\dots\ast\Delta'_{n}$.

The \emph{blown-up complex} of $X$ is the cochain complex 
$\tN^*(X;R)$
composed of the elements $\omega$
associating to each  regular simplex
 $\sigma\colon \Delta_{0}\ast\dots\ast\Delta_{n}\to X$
an element
 $\omega_{\sigma}\in  \tN^*_{\sigma}$  
such that $\delta_{\ell}^*(\omega_{\sigma})=\omega_{\partial_{\ell}\sigma}$,
for any  regular face operator
 $\delta_{\ell}\colon\Delta'\to\Delta$.
 The differential $d \omega$ is defined by
 $(d \omega)_{\sigma}=d(\omega_{\sigma})$.
 The \emph{perverse degree of $\omega$ along a singular stratum $S$} equals
 $$\|\omega\|_S=\sup\left\{
 \|\omega_{\sigma}\|_{\codim S}\mid \sigma\colon \Delta\to X \;
 \text{regular such that }
 \sigma(\Delta)\cap S\neq\emptyset
 \right\}.$$
 By setting $\|\omega\|_{S}=0$ for any regular stratum $S$, we get a map
 $\|\omega\|\colon \cS\to \N$.
 
 \begin{definition}\label{def:blowup}
 A \emph{cochain $\omega\in \tN^*(X;R)$  is $\ov{p}$-allowable} if $\| \omega\|\leq \ov{p}$
 and of  \emph{$\ov{p}$-intersection} if $\omega$ and $d\omega$ are $\ov{p}$-allowable. 
 We denote $\tN^*_{\ov{p}}(X;R)$
 the complex of $\ov{p}$-intersection cochains and  
 $\crH^*_{\ov{p}}(X;R)$ 
 its homology, called
 \emph{blown-up $\ov{p}$-intersection cohomology}  of $X$.
\end{definition}

Let us recall  its main properties.
First, the canonical projection 
 $\pr \colon X\times\mathbb{R}\rightarrow X$ induces  an isomorphism (\cite[Theorem D]{CST4})
\begin{equation}\label{pro}
\pr^* \colon \crH^*_{\ov{p}}(X;R)\rightarrow\crH^*_{\ov{p}}(X\times \R;R).
\end{equation}
Also, if $L$ is a compact pseudomanifold and  $\overline p$  a perversity on the cone $\rc L$,  inducing
$\ov{p}$ on $L$, we have  \cite[Theorem E]{CST4}:
\begin{equation}\label{Cone}
\crH^*_{\ov{p}}(\rc L;R)
=\begin{cases}
\crH^*_{\ov{p}}(L;R),
& \text{if }k\leq\ov p(\tv),\\
0 & \text{if }k>\ov p(\tv),
\end{cases}
\end{equation}
where  $\tv$ is the apex of the cone.
If $k\leq\ov p(\tv)$, the isomorphism
$\crH^*_{\ov{p}}(\rc L;R)\cong \crH^*_{\ov{p}}(L;R)$
 is given by the inclusion $L \times ]0,1[ = \rc L \backslash\{\tv\} \hookrightarrow \rc L$.

\begin{definition}\label{def:Upetit}
Let $\cU$ be an open cover of $X$.
A \emph{$\cU$-small simplex} is a regular simplex, $\sigma\colon \Delta=\Delta_{0}\ast\dots\ast\Delta_{n}\to X$,
 such that there exists $U\in\cU$ with $\im\sigma\subset U$. 
 The \emph{blown-up  complex of $\cU$-small cochains of $X$
 with coefficients in $R$,} written 
 $\tN^{*,\cU}(X;R)$
is the cochain complex made up of elements  $\omega$, associating to any $\cU$-small simplex,
 $\sigma\colon\Delta= \Delta_{0}\ast\dots\ast\Delta_{n}\to X$,
an element
 $\omega_{\sigma}\in \Hiru \tN*\Delta$,  
 so that $\delta_{\ell}^*(\omega_{\sigma})=\omega_{\partial_{\ell}\sigma}$,
for any face operator,
 $\delta_{\ell}\colon \Delta'_{0}\ast\dots\ast\Delta'_{n}\to \Delta_{0}\ast\dots\ast\Delta_{n}$, with $\Delta'_{n}\neq\emptyset$. 
 If $\ov{p}$ is a perversity on $X$, we denote by 
 $\tN^{*,\cU}_{\ov{p}}(X;R)$
 the complex of 
 $\cU$-small cochains verifying
$ \|\omega\|\leq \ov{p}$ and
$\|\delta \omega\|\leq\ov{p}$.
\end{definition}

\begin{proposition}{\cite[Corollary 9.7]{CST4}}\label{cor:Upetits}
The restriction map is a quasi-isomor\-phism,\newline
$\rho_{\cU}\colon \lau \tN{*}{\ov{p}}{X;R}
\to \lau\tN{*,\cU}{\ov{p}}{X;R}$.
\end{proposition}

Finally, the blown-up intersection cohomology satisfies the
  Mayer-Vietoris property.

\begin{proposition}{\cite[Theorem C]{CST4}}\label{thm:MVcourte}
Let $(X,\ov{p})$ be a paracompact perverse space, 
endowed with an open cover  $\cU =\{W_{1},W_{2}\}$
and a subordinated partition of the unity, $(f_{1},f_{2})$. 
For $i=1,\,2$, we denote by $\cU_{i}$ the cover of  $W_{i}$ consisting of the open subsets
$(W_{1}\cap W_{2}, f_{{i}}^{-1}(]1/2,1])$ and by $\cU$ the cover of  $X$, union of the covers  $\cU_{i}$.
Then, the canonical inclusions, $W_{i}\subset X$ and $W_{1}\cap W_{2}\subset W_{i}$, induce a short exact sequence,
where $\varphi(\omega_{1},\omega_{2})=\omega_{1}-\omega_{2}$,
$$
\xymatrix@C=4mm{
0\ar[r]&
\tN^{*,\cU}_{\ov{p}}(X;R)\ar[r]
&
\tN^{*,\cU_{1}}_{\ov{p}}(W_{1};R)
\oplus
\tN^{*,\cU_{2}}_{\ov{p}}(W_{2};R)
\ar[r]^-{\varphi}&
\tN^{*}_{\ov{p}}(W_{1}\cap W_{2};R)\ar[r]&
0.
}$$
\end{proposition}

\section{Borel-Moore intersection homology}\label{sec:BM}

In a filtered space $X$, locally finite chains are sums,  perhaps infinite,
$\xi=\sum_{j\in J}\lambda_{j}\sigma_{j}$,
such that every point $x\in X$ has a  neighborhood $U_{x}$
for which all but a finite number of the regular simplices $\sigma_{j}$
(see \defref{def:regularsimplex})  with support intersecting $U_{x}$
have a coefficient $\lambda_{j}$ equal to 0.

\begin{definition}\label{def:BM}
Let $(X,\ov{p})$ be a perverse space.
We denote by 
$\gC^{\infty,\ov{p}}_{*}(X;R)$
the complex of  locally finite   chains of $\ov{p}$-intersection
with the differential $\gd$.
Its homology, 
$\gH^{\infty,\ov{p}}_{*}(X;R)$,
is called \emph{the locally finite (or Borel-Moore) $\ov p$-intersection homology.}
\end{definition}

Recall  a characterization of locally finite   $\ov{p}$-intersec\-tion chains.

\begin{proposition}{\cite[Proposition 3.4]{CST5}}\label{prop:BMprojectivelim}
Let $(X,\ov{p})$ be a perverse space.
Suppose that $X$ is locally compact, metrizable and separable.
Then, the complex of locally finite  $\ov{p}$-intersection chains
 is isomorphic to the inverse limit of complexes,
 $$\gC^{\infty,\ov{p}}_{*}(X;R) \cong \varprojlim_{K\subset X}\gC^{\ov{p}}_{*}(X,X\menos K;R),
 $$
where the limit is taken over all compact subsets of $X$.
\end{proposition}

Since locally finite chains in an open subset of $X$ are not necessarily locally finite in $X$, the 
complex $\gC^{\infty,\ov{p}}_{*}(-;R)$ is not functorial for the inclusions of open subsets. To get round this defect, 
we introduce a contravariant chain complex as in \cite{Bre}. In the context of intersection homology a similar
approach is done by Friedman (see  \cite[Section 2.3.2]{MR2276609}), the sheafification of the resulting complex 
being nothing but Deligne's sheaf of \cite{GM2}. 
So we set
$$\gC^{\infty,X,\ov{p}}_{*}(U;R):=
 \varprojlim_{K\subset U} 
  \gC^{\ov{p}}_{*}(X,X\backslash K;R),
 $$
 where $K$ runs over the family of compact subsets of $U$.
An element   
$\alpha\in \gC^{\infty,X,\ov{p}}_{*}(U;R)$ 
is  a family $\alpha=\langle\alpha_K\rangle_{K}$, indexed by  the family of compacts of $U$, 
with $\alpha_K \in \gC_{*}^{\ov{p}}(X;R)$ and $\alpha_{K'}  - \alpha_{K} \in \gC_{*}^{\ov{p}}(X\menos K;R)$,
if $K \subset K'$. In particular, $\alpha=0$ if, and only if, $\alpha_K \in\gC_{*}^{\ov{p}}(X\menos K;R)$  
for every $K$.

\medskip
 For the construction of the projective limit, an exhaustive family of compacts suffices. Therefore,
if $X$ is hemicompact, we may use a numerable increasing sequence of compacts, $(K_{i})_{i\in\N}$, and get
$\alpha=\langle \alpha_{i}\rangle_{i}$,
with $\alpha_{i}\in \gC_{*}^{\ov{p}}(X;R)$ and $\alpha_{i+1}  - \alpha_{i} \in \gC_{*}^{\ov{p}}(X\menos K_{i};R)$.

\medskip
 Given two open subsets $V \subset U \subset X$, we denote by 
 \begin{equation}\label{equa:maps2}
 I^{X,\ov{p}}_{V,U}\colon
 \varprojlim_{K\subset U}
  \gC^{\ov{p}}_{*}(X,X\backslash K;R)
  \to  \varprojlim_{K\subset V}
   \gC^{\ov{p}}_{*}(X,X\backslash K;R)
 \end{equation}
  the map induced by the identity.
 So, the complex 
 $\gC^{\infty,X,\ov{p}}_{*}(-;R)$
 defines a contravariant functor  from the poset of open subsets  of $X$.
 Moreover, this is an appropriate substitute for the study of locally finite  $\ov p$-intersec\-tion homology
  as shows the following result.

\begin{proposition}\label{prop:BMprojectivelimU}
Let $(X,\ov{p})$ be a locally compact, second-countable 
perverse space and $U \subset X$ an open subset. The natural restriction 
$I_{U}^{\ov{p}}\colon 
\gC^{\infty,\ov{p}}_{*}(U;R)
\to
\gC^{\infty,X,\ov{p}}_{*}(U;R)$
is a quasi-isomorphism.
\end{proposition}

\begin{proof}
Let $(K_{i})_{i\in\N}$ be a numerable increasing sequence of compacts of $U$, covering $U$
and cofinal in the family of compact subsets of $U$.
The maps
$\gC_{*}^{\ov{p}}(U,U\menos K_{i})\to \gC_{*}^{\ov{p}}(U,U\menos K_{i+1})$
and
$\gC_{*}^{\ov{p}}(X,X\menos K_{i})\to \gC_{*}^{\ov{p}}(X,X\menos K_{i+1})$
being surjective, these two sequences verify the Mittag-Leffler condition.
Thus the inclusions 
$
(\gC_{*}^{\ov{p}}(U,U\menos K_{i}))_{i}
\to
(\gC_{*}^{\ov{p}}(X,X\menos K_{i}))_{i}
$
give a morphism of short exact sequences (\cite[Proposition 3.5.8]{MR1269324}):
$$\xymatrix@C=4mm{
0\ar[r]&
\varprojlim^1_{i}\gH_{k+1}^{\ov{p}}(U,U\menos K_{i})
\ar[d]\ar[r]&
H_{k}(\varprojlim_{i}\gC_{*}^{\ov{p}}(U,U\menos K_{i}))
\ar[d]\ar[r]&
\varprojlim_{i}\gH_{k}^{\ov{p}}(U,U\menos K_{i})
\ar[d]\ar[r]&
0\\
0\ar[r]&
\varprojlim^1_{i}\gH_{k+1}^{\ov{p}}(X,X\menos K_{i})
\ar[r]&
H_{k}(\varprojlim_{i}\gC_{*}^{\ov{p}}(X,X\menos K_{i}))
\ar[r]&
\varprojlim_{i}\gH_{k}^{\ov{p}}(X,X\menos K_{i})
\ar[r]&
0.
}$$
The result is now a consequence of the excision property.
\end{proof}

The existence of a Mayer-Vietoris exact sequence in this context can be deduced from
a sheaf theoretic argument in the case of perversities depending only on the codimension of strata,
as mentioned in  \cite[Proof of Proposition 2.20]{MR2276609}. 
We provide below a direct proof
for general perversities.

\begin{theorem}\label{thm:MV}
Let $(X,\ov{p})$ be a  locally compact, second-countable perverse space
and $\{U,V\}$ an open covering of $X$. Then we have a Mayer-Vietoris exact sequence,
with coefficients in $R$,
\begin{equation}\label{equa:MV}
\xymatrix@C=4.9mm{
\dots \ar[r] &
\gH^{\infty,\ov{p}}_{k}(X)
\ar[r] &
\gH^{\infty,\ov{p}}_{k}(V)
\oplus
\gH^{\infty,\ov{p}}_{k}(U)
\ar[r] &
\gH^{\infty,\ov{p}}_{k}(U\cap V)
\ar[r] &
\gH^{\infty,\ov{p}}_{k-1}(X)
\ar[r] &
\dots
}
\end{equation}
\end{theorem}

\begin{proof}
As $U$ and $V$ are hemicompact, we choose sequences $(U_{i})_{i\in\N}$ and $(V_{i})_{i\in\N}$ of 
relatively compact open subsets
of $U$ and $V$, respectively, such that $\ov{U}_{i}\subset U_{i+1}$, $\cup_{i\in\N}U_{i}=U$
and
$\ov{V}_{i}\subset V_{i+1}$, $\cup_{i\in\N}V_{i}=V$.
Let us notice that $(\ov{U}_{i}\cup \ov{V}_{i})_{i\in\N}$ and $(\ov{U}_{i}\cap \ov{V}_{i})_{i\in\N}$
are  sequences of compact subsets such that  
$\ov{U}_{i}\cup \ov{V}_{i}\subset \ov{U}_{i+1}\cup \ov{V}_{i+1}$
and
$\ov{U}_{i}\cap \ov{V}_{i}\subset \ov{U}_{i+1}\cap \ov{V}_{i+1}$ 
which are exhaustive for $U\cup V$ and $U\cap V$ respectively.

As  observed in the proof of \propref{prop:BMprojectivelimU}, the sequences
$(\gC_{*}^{\ov{p}}(X,X\menos K_{i}))_{i\in\N}$
satisfy the Mittag-Leffler property, for $K_{i}=\ov{U}_{i}$, $\ov{V}_{i}$, $\ov{U}_{i}\cap \ov{V}_{i}$ or 
$\ov{U}_{i}\cup \ov{V}_{i}$.
Therefore, the short exact sequences
$$0\to
\gC_{*}^{\ov{p}}(X\menos (\ov{U}_{i}\cup \ov{V}_{i}))
\to
\gC_{*}^{\ov{p}}( X\menos \ov{U}_{i})\oplus
\gC_{*}^{\ov{p}}(X\menos \ov{V}_{i})
\to
\gC_{*}^{\ov{p}}(X\menos \ov{U}_{i})+
\gC_{*}^{\ov{p}}(X\menos \ov{V}_{i})
\to
0
$$
induce the short exact sequence
\begin{equation}\label{MV2}
 \def\objectstyle{\scriptstyle}
 \xymatrix@C=4mm{
 0\ar[r]&
\varprojlim_{i}\gC_{*}^{\ov{p}}(X, X\menos (\ov{U}_{i}\cup \ov{V}_{i}))
\ar[r]&
\varprojlim_{i}\gC_{*}^{\ov{p}}(X, X\menos \ov{U}_{i})\oplus
\varprojlim_{i}
\gC_{*}^{\ov{p}}(X, X\menos \ov{V}_{i})
\ar[r]&
\varprojlim_{i} \gQ_{*}^{\ov{p}}(X, \ov{U}_{i},\ov{V}_{i}))
\ar[r]&
0}
\end{equation}
with
$ \gQ_{*}^{\ov{p}}(X, \ov{U}_{i},\ov{V}_{i}))
=\gC_{*}^{\ov{p}}(X)/(\gC_{*}^{\ov{p}}(X\menos \ov{U}_{i})+
\gC_{*}^{\ov{p}}(X\menos \ov{V}_{i}))$.
The  long exact sequence associated to \eqref{MV2} gives \eqref{equa:MV}. Let us see that.

\medskip
$\bullet$ First, by \propref{prop:BMprojectivelimU}, we have
$\gH^{\infty,\ov{p}}_{*}(U\cup V)
\cong 
H_{k}(\varprojlim_{i}\gC_{*}^{\ov{p}}(X, X\menos (\ov{U}_{i}\cup \ov{V}_{i})))$,
$\gH^{\infty,\ov{p}}_{*}(U)
\cong 
H_{k}(\varprojlim_{i}\gC_{*}^{\ov{p}}(X, X\menos \ov{U}_{i}))$
and
$\gH^{\infty,\ov{p}}_{*}( V)
\cong 
H_{k}(\varprojlim_{i}\gC_{*}^{\ov{p}}(X, X\menos  \ov{V}_{i}))$.

\medskip
$\bullet$ As for the \emph{third term} of the sequence (\ref{MV2}), we proved 
 in \cite[Proposition 4.1]{CST3} that the identity map on $X$ induces a quasi-isomorphism,
$\gC^{\ov{p}}_{*}(X\menos \ov{U}_{i}) + \gC^{\ov{p}}_{*}(X\menos \ov{V}_{i}) 
\to
\gC^{\ov{p}}_{*}(X\menos (\ov{U}_{i}\cap \ov{V}_{i}))$. 
Therefore, it induces  a quasi-isomorphism
$$ \gQ_{*}^{\ov{p}}(X, \ov{U}_{i},\ov{V}_{i}))
\to
\gC^{\ov{p}}_{*}(X,X\menos (\ov{U}_{i}\cap \ov{V}_{i})).
$$
With the Mittag-Leffler property and \propref{prop:BMprojectivelimU},
the identity  map also gives a quasi-isomorphism
\begin{equation}\label{equa:MV3}
\psi\colon
\varprojlim_{i} \gQ_{*}^{\ov{p}}(X, \ov{U}_{i},\ov{V}_{i})
\to
\varprojlim_{i} \gC^{\ov{p}}_{*}(X,X\menos (\ov{U}_{i}\cap \ov{V}_{i}))
=
\gC^{\infty,X,\ov{p}}_{*}(U\cap V).
\end{equation}
\end{proof}

The following properties have been proven in \cite{CST5}.

\begin{proposition}{\cite[Proposition 3.5]{CST5}}\label{prop:foisR}
Let $(L,\ov{p})$ be a compact perverse space. Then we have
$$\gH^{\infty,\ov{p}}_{k}(\R^m\times L;R)
=
\gH_{k-m}^{\ov{p}}(L;R).$$
\end{proposition}

\begin{proposition}{\cite[Proposition 3.7]{CST5}}\label{prop:conefoisR}
Let $L$ be a compact space
 and $\ov{p}$ be a perversity on the cone $\rc L$ of apex  $\tv$. Then we have
 $$\gH^{\infty,\ov{p}}_{k}(\R^m\times \rc L;R)=\left\{
\begin{array}{ccl}
0&\text{if}&k\leq m+D\ov{p}(\tv)+1,\\
\gH^{\ov{p}}_{k-m-1}(L;R)
&\text{if}&k\geq m+ D\ov{p}(\tv)+2.
\end{array}\right.
 $$
\end{proposition}

\section{Poincar\'e duality}\label{sec:Poincare}

\subsection{Fundamental class and cap product}\label{subsec:cap}
Let $(X,\ov{p})$ be a perverse pseudomanifold of dimension $n$.

Recall from \cite{GM1} that an $R$-orientation of  $X$
is an $R$-orientation of the manifold $X^n:=X\menos X_{n-1}$. For any $x\in X^n$, we denote by
$\tto_{x}\in H_{n}(X^n,X^n\menos\{x\};R)=\gH_{n}^{\ov{0}}(X,X\menos\{x\};R)$ the associated local orientation.
We know (see \cite{FM} or \cite[Theorem 8.1.18]{FriedmanBook}) that, 
for any compact $K\subset X$, there exists a unique element
$\Gamma^{K}_{X}\in \gH_{n}^{\ov{0}}(X,X\menos K;R)$
whose restriction equals $\tto_{x}$ for any $x\in K$.
These classes give a Borel-Moore homology class, called \emph{the fundamental class of $X$,}
$$\Gamma_{X}=\langle \Gamma^{K}_{X}\rangle_{K}\in
\gH^{\infty,\ov{0}}_{n}(X;R).$$
The fundamental classes are natural for the injections between open subsets of $X$.
 Given two open subsets $V \subset U \subset X$,  the map induced in Borel-Moore homology by the identity 
 \begin{equation}\label{equa:restr}
 I^{X,\ov{p}}_{V,U}\colon 
 \varprojlim_{K\subset U}
  \gC^{\ov{p}}_{*}(X,X\backslash K;R)
  \to  \varprojlim_{K\subset V}
   \gC^{\ov{p}}_{*}(X,X\backslash K;R)
 \end{equation}
 sends $\Gamma_{U}$ on $\Gamma_{V}$, see \cite[Theorem 8.1.18]{FriedmanBook}.

\medskip
Suppose that $X$ is equipped with two perversities, $\ov{p}$ and $\ov{q}$.
 In \cite[Proposition 4.2]{CST4}, we prove the existence of a map 
 \begin{equation}\label{equa:cupprduitespacefiltre}
 -\smile -\colon\tN^k_{\ov{p}}(X;R)\otimes \tN^{\ell}_{\ov{q}}(X;R)\to \tN^{k+\ell}_{\ov{p}+\ov{q}}(X;R),
 \end{equation}
 inducing  an associative  and  commutative graded product, called \emph{intersection cup product,}
  \begin{equation}\label{equa:cupprduitTWcohomologie}
-\smile -\colon \crH^k_{\ov{p}}(X;R)\otimes \crH^{\ell}_{\ov{q}}(X;R)\to
 \crH^{k+\ell}_{\ov{p}+\ov{q}}(X;R).
 \end{equation}
 Mention also from \cite[Propositions 6.6 and 6.7]{CST4} the existence of \emph{cap products,}
\begin{equation}\label{equa:caphomology}
 -\frown - \colon \widetilde{N}_{\overline{p}}^{i}(X;R)\otimes {\gC}_{j}^{\overline{q}}(X;R)
\to {\gC}_{j-i}^{\overline{p}+\overline{q}}(X;R)
\end{equation}
such that $
(\eta\smile \omega)\frown \xi=\eta\frown(\omega\frown\xi)$ and
$\gd(\omega\frown \xi)=d\omega\frown\xi+(-1)^{|\omega|}\omega\frown \gd \xi
$.
Thus, this cap product induces a \emph{cap product} in homology,
\begin{equation}\label{equa:caphomologyhomology}
 - \frown - \colon \crH_{\overline{p}}^{i}(X;R)\otimes {\gH}_{j}^{\overline{q}}(X;R)
\to {\gH}_{j-i}^{\overline{p}+\overline{q}}(X;R).
\end{equation}
The map \eqref{equa:caphomology} can be extended to a map 
\begin{equation}\label{equa:capBM}
 -\frown - \colon \widetilde{N}_{\overline{p}}^{i}(X;R)\otimes {\gC}_{j}^{\infty,\overline{q}}(X;R)
\to {\gC}_{j-i}^{\infty,\overline{p}+\overline{q}}(X;R)
\end{equation}
as follows:\\
given $\alpha\in \tN^i_{\ov{p}}(X;R)$
and
$\eta=\langle \eta_{K}\rangle_{K}\in \gC^{\infty,\ov{q}}_{j}(X;R)$, we set
$\alpha\frown \eta=\langle \alpha\frown \eta_{K}\rangle_{K}\in  \gC^{\infty,\ov{p}+\ov{q}}_{j-i}(X;R)$.
This definition makes sense since the cap product \eqref{equa:caphomology} is natural. Moreover, from 
the compatibility with the differentials, we get an induced map, 
\begin{equation}\label{equa:caphomologyhomologyBM}
 - \frown - \colon \crH_{\overline{p}}^{i}(X;R)\otimes {\gH}_{j}^{\infty,\overline{q}}(X;R)
\to {\gH}_{j-i}^{\infty,\overline{p}+\overline{q}}(X;R).
\end{equation}
As  $\Gamma_{X}\in {\gH}_{n}^{\infty,\overline{0}}(X;R)$, the cap product with the fundamental class gives a map,
\begin{equation}\label{equa:dualmap}
\cD_{X}:=-\frown \Gamma_{X}\colon \crH^*_{\ov{p}}(X;R)\to \gH^{\infty,\ov{p}}_{n-*}(X;R),
\end{equation}
that is the Poincar\'e duality map of the next theorem. 
Let us emphasize that this map exists at the level of chain complexes.

\medskip
In this paradigm, the Poincar\'e duality comes from a cap product with the fundamental class, $\Gamma_{X}$. 
As this one is
a Borel-Moore homology class, we need to adapt \eqref{equa:caphomologyhomology}. In fact, there are two ways of working.
\begin{itemize}
\item We may keep the blown-up cohomology, for which the cap  with $\Gamma_{X}$ gives a Borel-Moore homology class as in  \eqref{equa:dualmap}. This is the subject of this work which brings \thmref{thm:dual} below.
\item We may also work with  blown-up cohomology with compact supports, for which the cap product
with $\Gamma_{X}$ gives a (finite) homology class. That is, we extend \eqref{equa:caphomology} in
\begin{equation}\label{equa:capcompact}
 -\frown - \colon \widetilde{N}_{\overline{p},c}^{i}(X;R)\otimes {\gC}_{j}^{\infty,\overline{q}}(X;R)
\to {\gC}_{j-i}^{\overline{p}+\overline{q}}(X;R),
\end{equation}
This approach was used in \cite{CST2} and led  to the Poincar\'e duality
(\cite[Theorem B]{CST2})
$$-\frown \Gamma_{X}\colon \crH^{i}_{\ov{p},c}(X;R)\xrightarrow{\cong}
\gH_{n-i}^{\ov{p}}(X;R).$$
\end{itemize}

\subsection{Main theorem} 

We prove the existence of an isomorphism between the Borel-Moore $\ov{p}$-intersection homology 
and the blown-up 
$\ov{p}$-intersection cohomology.

\begin{theorem}\label{thm:dual}
Let $(X,\ov p)$ be an $n$-dimensional,  second countable and oriented perverse pseudomanifold.  
The cap product with the fundamental class induces a Poincar\'e duality isomorphism
$$\cD_{X}\colon \crH^*_{\ov{p}}(X;R)
\xrightarrow{\cong}
 \gH^{\infty,\ov{p}}_{n-*}(X;R).$$
\end{theorem}

The proof  uses the following result. 

\begin{proposition}{\cite[Proposition 13.2]{CST5}}\label{prop:supersuperbredon}
Let $\mathcal F_{X}$ be the category whose objects are (stratified homeomorphic to) open subsets
of a given paracompact and separable CS set $X$ and whose morphisms are  stratified homeomorphisms and inclusions.
Let  $\mathcal Ab_{*}$ be the category of graded abelian groups. Let $F^{*},\,G^{*}\colon \mathcal F_{X}\to \mathcal Ab$
be two functors and
 $\Phi\colon F^{*}\to G^{*}$ a natural transformation satisfying
 the conditions listed
below.
\begin{enumerate}[(i)]
\item The functors $F^{*}$ and $G^{*}$ admit  Mayer-Vietoris exact sequences and the natural transformation $\Phi$ 
 induces a commutative diagram between these sequences.
\item If $\{U_{\alpha}\}$ is a disjoint collection of open subsets of $X$  and $\Phi\colon F_{*}(U_{\alpha})\to G_{*}(U_{\alpha})$ is an isomorphism for each $\alpha$, then $\Phi\colon F^{*}(\bigsqcup_{\alpha}U_{\alpha})\to G^{*}(\bigsqcup_{\alpha}U_{\alpha})$  is an isomorphism.
\item If $L$ is a compact filtered space such that 
$X$  has an open subset  stratified homeomorphic
to $\R^i\times \rc L$ and, if
$\Phi\colon F^{*}(\R^i\times (\rc L\backslash \{\tv\}))\to G^{*}(\R^i\times (\rc L\backslash \{\tv\}))$
is an isomorphism, then so is
$\Phi\colon F^{*}(\R^i\times \rc L)\to G^{*}(\R^i\times \rc L)$. (Here, $\tv$ is the apex of the cone $\rc L$.)
\item If $U$ is an open subset of X contained within a single stratum and homeomorphic
to an Euclidean space, then $\Phi\colon F^{*}(U)\to G^{*}(U)$ is an isomorphism.
\end{enumerate}
Then $\Phi\colon F^{*}(X)\to G^{*}(X)$ is an isomorphism.
\end{proposition}

\begin{proof}[Proof of \thmref{thm:dual}]
As any open subset $U\subset X$ is an oriented pseudomanifold, we may consider the 
associated homomorphism defined in (\ref{equa:dualmap}), 
$\dos {\cD}{U}\colon  \crH_{k}^{ \ov{p}}(U) \to   \gH^{\infty,X,\ov{p}}_{n-k}(U)$,
where we use the identification
$\gH^{\infty,\ov{p}}_{n-k}(U)
\xrightarrow{\cong}
\gH^{\infty,X,\ov{p}}_{n-k}(U)$
given by \propref{prop:BMprojectivelimU}.
Let $V\subset U\subset X$ be two open subsets of $X$, endowed with the induced structures of pseudomanifold
of $X$. The canonical inclusion (see \eqref{equa:maps2}) and 
the cap product with the fundamental class give a commutative diagram,

\begin{equation}\label{conm}
\xymatrix{
\crH^k_{\ov{p}}(U)
\ar[d]_{\cD_{U}}
\ar[r]^-{(I^{\ov{p}}_{V,U})^*}
&
\crH^k_{\ov{p}}(V)
\ar[d]^{\cD_{V}}\\
\gH_{n-k}^{\infty,X,\ov{p}}(U)
\ar[r]^-{(I^{X,\ov{p}}_{V,U})^*}
&
\gH_{n-k}^{\infty,X,\ov{p}}(V).
}
\end{equation}
We apply \propref{prop:supersuperbredon} to the natural transformation
$\cD_{U}\colon \crH_{\ov{p}}^k(U)
\to
\gH_{n-k}^{\infty,X,\ov{p}}(U)$.
The proof is reduced to the verifications of its hypotheses. 

\medskip
$\bullet$ First, let us notice that the conditions (ii) and (iv) are direct.

\medskip
$\bullet$ \emph{Property} (i). Let $\cU = \{W_1,W_2\}$ be an open covering of $X$.
Mayer-Vietoris sequences are constructed in \propref{thm:MVcourte} and \thmref{thm:MV}. We
 build a morphism between them with the following diagram.

 In the first row, we take over the notations of \propref{thm:MVcourte}.
As in the proof of \thmref{thm:MV}, we choose sequences $(U_{i})_{i\in\N}$ and $(V_{i})_{i\in\N}$ of 
relatively compact open subsets
of $W_{1}$ and $W_{2}$, respectively, such that $\ov{U}_{i}\subset U_{i+1}$, $\cup_{i\in\N}U_{i}=W_{1}$
and
$\ov{V}_{i}\subset V_{i+1}$, $\cup_{i\in\N}V_{i}=W_{2}$.
The last row corresponds to \eqref{MV2}.
$$
\def\objectstyle{\scriptstyle}
\xymatrix@C=3.6mm{		
0\ar[r]&
\tN^{*,\cU}_{\ov{p}}(X)
\ar[r] &
\tN^{*,\cU_{1}}_{\ov{p}}(W_{1})
\oplus
\tN^{*,\cU_{2}}_{\ov{p}}(W_{2})
\ar[r]  
\ar@{}[dl]+<10pt>^{\fbox{\tiny{I}}}&
\tN^*_{\ov{p}}(W_{1}\cap W_{2})
\ar[r]  \ar@{}[dl]+<10pt>^{\fbox{\tiny{II}}}&
0\\
&
\tN^*_{\ov{p}}(X)
\ar[r] \ar[u]^{\rho}   
\ar[d]_{\frown \gamma_X}&
\tN^*_{\ov{p}}(W_{1})
\oplus
\tN^*_{\ov{p}}(W_{2})
 \ar[r]   \ar[u]_{\rho}  \ar@{}[dl]+<10pt>^{\fbox{\tiny{III}}} 
\ar[d]^{\frown \gamma_{W_1} \oplus \frown \gamma_{W_2}}&
\tN^*_{\ov{p}}(W_{1}\cap W_{2})
\ar@{=}[u] 
 \ar@{}[dl]+<10pt>^{\fbox{\tiny{IV}}}
\ar[d]^{\frown \gamma_{W_1 \cap W_2}}
&&\\
&\gC^{\infty,\ov{p}}_{n-*}(X)
 \ar[r]  
 \ar[d]_{I^{\ov{p}}_{X}}&
\gC^{\infty,\ov{p}}_{n-*}(W_{1})
\oplus
\gC^{\infty,\ov{p}}_{n-*}(W_{2})
\ar[r]     
\ar@{}[dl]+<10pt>^{\fbox{\tiny{V}}} 
\ar[d]^{I^{\ov{p}}_{W_{1}}\oplus {I^{\ov{p}}_{W_{2}}}}&
\gC^{\infty,\ov{p}}_{n-*}(W_{1}\cap W_{2})
\ar@{}[dl]+<10pt>^{\fbox{\tiny{VI}}}  
\ar[d]^{I^{\ov{p}}_{W_{1}\cap W_{2}}}
&&\\
&\gC^{\infty,X,\ov{p}}_{n-*}(X) 
\ar[r] &
\gC^{\infty,X,\ov{p}}_{n-*}(W_{1})
\oplus
\gC^{\infty,X,\ov{p}}_{n-*}(W_{2}) 
\ar[r]     
\ar@{}[dl]+<10pt>^{\fbox{\tiny{VII}}}&
\gC^{\infty,X,\ov{p}}_{n-*}(W_{1}\cap W_{2})
\ar@{}[dl]+<10pt>^{\fbox{\tiny{VIII}}}&&\\
 0\ar[r]&
\varprojlim_{i}\gC_{*}^{\ov{p}}(X, X\menos (\ov{U}_{i}\cup \ov{V}_{i}))
\ar@{=}[u]
\ar[r]&
\varprojlim_{i}\gC_{*}^{\ov{p}}(X, X\menos \ov{U}_{i})\oplus
\varprojlim_{i}\gC_{*}^{\ov{p}}(X, X\menos \ov{V}_{i})
\ar@{=}[u]
\ar[r]&
\varprojlim_{i} \gQ_{*}^{\ov{p}}(X, \ov{U}_{i},\ov{V}_{i}))
\ar[u]_{\psi}
\ar[r]&0.
}
$$
%
The maps $\rho$ denote restriction and  the cycles $\gamma_{Y}\in \gC^{\infty,\ov{0}}_{n}(Y)$  
represent the fundamental class of $Y$, for $Y=X,\,W_{1},\,W_{2}, W_{1}\cap W_{2}$.
First, we prove that the above diagram commutes.

The square VII is clearly commutative. 
The vertical maps of squares I, II, V and VI  are induced by restrictions. Thus these squares are commutative. 
By naturality of the fundamental classes, we may choose for $\gamma_{W_1}$  the restriction of $\gamma_X$ 
and similarly for $\gamma_{W_2}$ and $\gamma_{W_1\cap W_2}$. 
So, diagrams III and IV commute.
 The diagram  VIII is commutative by construction of the map $\psi$, see \eqref{equa:MV3}.

\medskip
Let us observe that the columns of the diagrams I, II, V, VI, VII and VIII are quasi-isomorphisms. This is
a consequence of  
\cite[Theorem B]{CST4}, 
\propref{prop:BMprojectivelimU} and \eqref{equa:MV3} in the proof of \thmref{thm:MV}, respectively.
So, we get the following commutative diagram,
$$
 \def\objectstyle{\scriptstyle}
\xymatrix@C=6mm{	
\dots\ar[r]&
\crH^k_{\ov{p}}(X)
\ar[r] \ar[d]^{\cD_X}&
\crH^k_{\ov{p}}(W_{1})
\oplus
\crH^k_{\ov{p}}(W_{2})
 \ar[r]  \ar[d]^{\cD_{W_1} \oplus \cD_{W_2}}
&
\crH^k_{\ov{p}}(W_{1}\cap W_{2})
\ar[r] \ar[d]^{\cD_{W_1 \cap W_2}} 
&
\crH^{k+1}_{\ov{p}}(X)
\ar[r] \ar[d]^{\cD_X} 
 & 
 \dots 
\\
\dots 
\ar[r] &
\gH_{n-k}^{\infty,\ov{p}}(X) 
\ar[r] &
\gH_{n-k}^{\infty,\ov{p}}(W_{1})\oplus \gH_{n-k}^{\infty,\ov{p}}(W_{2)}
 \ar[r] 
&
\gH_{n-k}^{\infty,\ov{p}}(W_{1}\cap W_{2})
 \ar[r]  &
 \gH_{n-k-1}^{\infty,\ov{p}}(X)
  \ar[r]&
  \dots.
}
$$

\emph{$\bullet$ Property}  (iii) We apply    \eqref{pro}, \eqref{Cone}, 
\propref{prop:foisR}
 and
\propref {prop:conefoisR}.
First, we  have the isomorphism 
$
\crH^k_{\ov{p}}(\R^i\times \rc L)
= 0 =
\gH_{n-k}^{\infty,\ov{p}}(\R^i\times \rc L)$
for $k > \ov p(\tv)$, or equivalently $n-k < i + D \ov p(\tv) +2$. 
(Observe that $D\ov{p}(\tv)=n-i-2-\ov{p}(\tv)$.)
Next, let $k \leq \ov p(\tv)$.
The following commutative diagram comes from \eqref{conm}.
$$
\xymatrix{	
\crH^k_{\ov{p}}(\R^i\times \rc L)
\ar[r] \ar[d]_{\cD_{\R^i \times \rc L }}&
\crH^k_{\ov{p}}(\R^i \times L \times ]0,1[)  
\ar[d]^{\cD_{\R^i \times L \times ]0,1[ }}
\\
\gH_{n-k}^{\infty,\ov{p}}(\R^i \times \rc L)   
\ar[r] &
\gH_{n-k}^{\infty,\ov{p}}(\R^i \times  L \times ]0,1[).
}
$$
The right  column is an isomorphism by hypothesis. Since horizontal rows are also isomorphisms,  we 
deduce that $\cD_{\R^i \times \rc L }$ is an isomorphism. This proves Property  (iii).
 \end{proof}
 
 \begin{corollary}\label{cor:intersectionproduct}
 Let $(X,\ov p)$ be an $n$-dimensional,  second countable and oriented perverse pseudomanifold.  
 The Borel-Moore intersection homology can be endowed with an intersection product, induced from
 Poincar\'e duality and a cup product.  
 \end{corollary}
 
 \begin{proof}
 The following commutative diagram, whose vertical maps are isomorphisms, defines 
 the intersection product $\pitchfork$ from the cup product,
 \begin{equation}\label{equa:fork}
 \xymatrix{
\crH^k_{\ov{p}}(X;R)\otimes \crH^{\ell}_{\ov{q}}(X;R)
\ar[r]^-{-\smile-}\ar[d]^{\cong}\ar[d]_{\cD_X\otimes \cD_X}&
\crH^{k+\ell}_{\ov{p}+\ov{q}}(X;R)
\ar[d]^{\cD_X}\ar[d]_{\cong}
\\
\gH_{n-k}^{\infty,\ov{p}}(X;R)\otimes \gH_{n-\ell}^{\infty,\ov{q}}(X;R)
\ar[r]^-{-\pitchfork-}&
\gH_{n-k-\ell}^{\infty,\ov{p}+\ov{q}}(X;R).
}
\end{equation}
 \end{proof}

\begin{remark}
If $\ov{p}\colon \N\to \Z$ is a loose perversity as in \cite{MR2276609}, we have an inclusion of complexes
$\iota\colon
\gC_{*}^{\infty,\ov{p}}(X;R)
\hookrightarrow
I^{\ov{p}} \gC_{*}^{\infty}(X;R)$,
where $I^{\ov{p}} \gC_{*}^{\infty}(X;R)$
denotes the chain complex studied by Friedman in \cite{MR2276609}.
With the  technique of the  proof  of \thmref{thm:dual}, using \propref{prop:supersuperbredon}, we  also deduce that
$\iota$ is a quasi-isomorphism.
For the complex $I^{\ov{p}} \gC_{*}^{\infty}(X;R)$, the
existence of a Mayer-Vietoris exact sequence  
follows from the fact that its  homology proceeds from sheaf theory
and the computations involving  a cone  are done in
\cite[Propositions 2.18 and 2.20]{MR2276609}.
Thus, the blown-up intersection cohomology is Poincar\'e dual to the 
Borel-Moore intersection homology defined in \cite{MR2276609}, for any commutative ring of coefficients.
\end{remark}

\begin{remark}\label{rem:GJE}
 An intersection cohomology, denoted $\gH_{\ov{p}}^{*}(X;R)$, can also be defined from the linear dual,
 $\gC_{\ov{p}}^{*}(X;R)=\hom(\gC^{\ov{p}}_{*}(X;R),R)$. 
 This is the point of view adopted in  \cite{FriedmanBook}, \cite{FR1}, \cite{FM}.
 If $(X,\ov{p})$ is a locally $(\ov{p},R)$-torsion free perverse space $X$, this cohomology is isomorphic to the blown-up
 cohomology with the complementary perversity,
  $\gH_{\ov{p}}^{*}(X;R)\cong \crH^*_{D\ov{p}}(X;R)$, see \cite[Theorem F]{CST4}.
 (We send the reader to \cite{GM1} for the condition on the torsion, 
 noting that it is always satisfied if $R$ is a field.)
 Therefore, in this case, there is a Poincar\'e duality,
 $\gH_{\ov{p}}^{k}(X;R)\cong  \gH_{n-k}^{\infty,D\ov{p}}(X;R)$,
 induced by a cup product. 
 But, in contrast to the blown-up cohomology situation, such isomorphism does not exist in general
 as shows the following example.
 \end{remark}

 \begin{example}\label{exam:pasdual}
 Let  $X=\Sigma\R P^3\backslash{*}$ be the complementary subspace of a regular point in the suspension of
 the real projective space $\R P^3$, together with the constant perversity on 1, denoted $\ov{1}$.
 (Observe $\ov{1}=D\ov{1}$.)
 Direct calculations from the Mayer-Vietoris exact sequences give as only non zero values
  for the  homology and cohomologies,
 \begin{enumerate}[(i)]
 \item $\crH^0_{\ov{1}}(X;\Z)= \Z$ and $\crH^3_{\ov{1}}(X;\Z)= \Z_{2}$,
 \item $\gH_{1}^{\infty,\ov{1}}(X;\Z)=\Z_{2}$ and $\gH_{4}^{\infty,\ov{1}}(X;\Z)=\Z$,
 \item $\gH^{0}_{\ov{1}}(X;\Z)=\Z$ and $\gH^{1}_{\ov{1}}(X;\Z)=\Z_{2}$.
 \end{enumerate}
 We notice $  \crH^k_{\ov{1}}(X;R)
\cong
 \gH^{\infty,\ov{1}}_{4-k}(X;R)$ 
 as stated in \thmref{thm:dual}
 but $\gH^{1}_{\ov{1}}(X;\Z)=\Z_{2}\not\cong \gH_{3}^{\infty,\ov{1}}(X;\Z)=0$.
 \end{example}

\begin{remark}\label{rem:lastone}
 In the case of a compact oriented pseudomanifold, a diagram as 
 \eqref{equa:fork}
 has already been introduced in \cite[Section~4]{CST7} for the definition of a product in intersection homology,
 for any ring of coefficients.
 For a PL pseudomanifold, a product of geometric cycles is defined by Goresky and MacPherson
 (\cite{GM1})
 for a simplicial intersection homology.
 The relationship between \eqref{equa:fork} and the geometric product can be specified as below.
  (For sake of simplicity,
 we consider coefficients in a field.)
 \begin{itemize}
 \item From a diagram similar to \eqref{equa:fork}, Friedman and McClure
 define an intersection product  on the intersection homology of a PL pseudomanifold, via a duality isomorphism
 (\cite{FM})
 with the intersection cohomology $\gH_{\ov{p}}^{*}(X)$ recalled in \remref{rem:GJE}. 
 In \cite[Theorem 1.3]{2018arXiv181210585F},
 they connect it with the geometric product of \cite{GM1}.
 \item As recalled in  \remref{rem:GJE}, the blown-up cohomology studied in this paper is connected to 
 $\gH_{\ov{p}}^{*}(X)$ through a natural quasi-isomorphism, which is compatible with the algebra structures
 (see \cite[Corollary 4.4 and end of Section 4]{CST6}).
 \end{itemize}
 Consequently, the intersection product on $\gH_{\ast}^{\bullet}(X;R)$ defined by \eqref{equa:fork} 
 coincides with the original geometric product
 of Goresky and McPherson in the PL compact case.
 
 \medskip
 Finally, mention the existence of a construction of the intersection product via a Leinster partial algebra structure,
 done by Friedman in \cite{MR3857199}, who extends to the perverse setting the construction made by McClure
 on the chain complex of a PL manifold, see \cite{MR2255502}.
 \end{remark}
 
\providecommand{\bysame}{\leavevmode\hbox to3em{\hrulefill}\thinspace}
\providecommand{\MR}{\relax\ifhmode\unskip\space\fi MR }
\providecommand{\MRhref}[2]{%
  \href{http://www.ams.org/mathscinet-getitem?mr=#1}{#2}
}
\providecommand{\href}[2]{#2}

\end{document}